\documentclass{amsart}

\usepackage{amssymb}
\usepackage{amsmath}
\usepackage{amsthm}

\newtheorem{theorem}{Theorem}[section]
\newtheorem {prop}[theorem]{Proposition}
\newtheorem {cory}[theorem]{Corollary}
\newtheorem {lem}[theorem]{Lemma}

\theoremstyle{definition}
\newtheorem{df}[theorem]{Definition}

\theoremstyle{remark}
\newtheorem{rem}[theorem]{Remark}

\numberwithin{equation}{section}

\newcommand{\wass}{\widetilde{w}}

\newcommand{\wasp}{\widetilde{w}_p}
\newcommand{\wsecasp}{\widetilde{u}_p}
\newcommand{\hol}{\mathcal{H}ol}

\newcommand{\Cbb}{\mathbb C}
\newcommand{\Dbb}{\mathbb D}
\newcommand{\Nbb}{\mathbb N}
\newcommand{\zz}{{\mathbb Z}_+}
\newcommand{\Rbb}{\mathbb R}

\newcommand{\igl}{J}
\newcommand{\dom}{\mathcal{D}}
\newcommand{\rmax}{R}
\newcommand{\we}{w}

\newcommand{\wsec}{u}
\newcommand{\haw}{H_w}
\newcommand{\mm}{M}
\newcommand{\diff}{D}

\begin{document}

\title[Integration and differentiation operators]
{Integration and differentiation operators between growth spaces}

\author{Evgueni Doubtsov}

\address{St.~Petersburg Department
of V.A.~Steklov Mathematical Institute,
Fontanka 27, St.~Petersburg 191023, Russia}

\email{dubtsov@pdmi.ras.ru}

\thanks{This research was supported by the Russian Science Foundation (grant No. 18-11-00053).}


\date{}

\begin{abstract}
For arbitrary radial weights $w$ and $u$,
we study the integration operator between the growth spaces $H_w^\infty$
and $H_u^\infty$ on the complex plane. 
Also, we investigate the differentiation operator on the Hardy growth spaces $H_w^p$, $0<p<\infty$,
defined on the unit disk or on the complex plane.
As in the case $p=\infty$, the log-convex weights $w$ play a special role in the problems under consideration.
\end{abstract}

\maketitle

\section{Introduction}\label{s_int}

Let $\dom$ denote the complex plane $\Cbb$ or the unit disk $\Dbb$ of $\Cbb$.
Put $\rmax(\Dbb)=1$ and $\rmax(\Cbb)=\infty$.
We say that $\we: [0, \rmax(\dom))\to (0, +\infty)$
is a weight if $w$ is non-decreasing, continuous and unbounded.
If $\we$ is a weight on $[0, +\infty)$, then we additionally assume that $\log r = o (\log \we(r))$ as
$r\to +\infty$ to exclude the weights of polynomial growth.

Let $\hol(\dom)$ denote the space of holomorphic functions on $\dom$.
Given $0< p \le \infty$ and a weight $\we$ on $[0, \rmax(\dom))$, the Hardy growth space $\haw^p(\dom)$
consists of those $f\in\hol(\Dbb)$ for which
\[
\|f\|_{\haw^p} =\sup_{0\le r <\rmax(\dom)} \frac{\mm_p(f, r)}{\we(r)} < \infty,
\]
where
\begin{align*}
  \mm_p^p(f, r) &= \frac{1}{2\pi} \int_0^{2\pi} |f(r e^{i\theta})|^p \, d\theta, \quad 0<p<\infty,
\\
  \mm_\infty(f, r) &= \max_{0\le\theta<2\pi} |f(r e^{i\theta})|.
\end{align*}
Observe that different names have been used for $\haw^\infty(\dom)$.
In particular, $\haw^\infty(\dom)$ are often called weighted Banach spaces of analytic functions;
see, for example, \cite{AT17CM, AT17, BBT98, HL08}. Also, $\haw^\infty(\dom)$ are said to be growth spaces (see \cite{AD17}).

This paper is motivated by results obtained in \cite{AT17CM, AT17, HL08} on differentiation and integration
in the spaces $\haw^\infty(\dom)$.
In particular, we investigate the integration operator $\igl$ defined on $\hol(\dom)$ by the following identity:
\[
\igl f(z) = \int_0^z f(w)\, dw, \quad z\in \dom.
\]
The bounded and compact operators $\igl: H_\we^\infty(\Cbb) \to H_\wsec^\infty(\Cbb)$
are studied in \cite{AT17CM, AT17} under specific restrictions on $\we$ or for $\we=\wsec$.
We show that the problems in question have quite elementary solutions if
$\we$ is assumed to be equivalent to a power series with positive coefficients.
Combining this fact and a result from \cite{AD17}, we obtain, in particular, a
characterization of the corresponding compact operators for arbitrary weights $\we$ and $\wsec$; see Theorem~\ref{t_igl_cmp} below.

The main technical tool in the studies of various operators on $\haw^\infty(\dom)$ is the associated weight $\wass$ (see \cite{BBT98})
defined as
\[
\wass(t) = \sup\{\mm_\infty(f, t): \mm_\infty(f, r)\le \we(r)\ \textrm{for all}\ 0\le r<\rmax(\dom)\}.
\]

\begin{theorem}\label{t_igl_cmp}
Let $\we, \wsec$ be weights on $[0, \infty)$.
Then the following properties are equivalent:
\begin{itemize}
  \item [(i)] $$\lim_{r\to\infty} \frac{1}{\wsec(r)}\int_0^r \wass(t)\, dt  =0;$$
  \item [(ii)] $\igl: H_\we^\infty(\Cbb) \to H_\wsec^\infty(\Cbb)$ is a compact operator.
\end{itemize}
\end{theorem}

Next, for $0<p<\infty$, we introduce a notion of $p$-associated weight $\wasp$.
In fact, using $p$-associated weights, we study the differentiation operator $D$ defined on $\hol(\dom)$
and show that several results from \cite{AT17} are extendable
to $\haw^p(\dom)$ with $1\le p < \infty$ and partially with $0<p<\infty$.
Also, we obtain related results for the integration operator $J$.

\subsection*{Organization of the paper}
In Section~\ref{s_igl}, we study the integration operator on the growth spaces $H_\we^\infty(\dom)$
with emphasis on the case $\dom=\Cbb$;
in particular, we prove Theorem~\ref{t_igl_cmp}.
Section~\ref{s_assoc} is devoted to $p$-associated weights and their basic properties.
General conditions related to the boundedness of the differentiation operator $D: \haw^p(\dom) \to H_\wsec^p(\dom)$
are presented in Section~\ref{s_basic}.
Finally, as in \cite{AT17, HL08}, we consider in Section~\ref{s_special} the following specific situations:
$\we(r) = (1-r)\wsec(r)$ for $0\le r <1$ and $\we(r) = \wsec(r)$ for $0\le r <+\infty$.

\section{Integration operator on $H_\we^\infty(\dom)$}\label{s_igl}
As mentioned in the introduction, the bounded or compact integration operators between
growth spaces $H_\we^\infty(\Cbb)$ and $H_\wsec^\infty(\Cbb)$ are characterized in \cite{AT17CM}
under specific restrictions on $\we$ or for $\we=\wsec$.
So, in the present section, we are primarily interested in the case $\dom=\Cbb$.

\subsection{Bounded operators}

Given $u, v: X \to (0, +\infty)$, we say that $u$ is equivalent to $v$ (in brief, $u\asymp v$)
if there exist constants $C_1, C_2 >0$ such that
\[
C_1 u(x) \le v(x) \le C_2 u(x), \quad x\in X.
\]

\begin{prop}\label{p_igl_bdd}
Let $\we, \wsec$ be weights on $[0, \rmax(\dom))$.
Assume that $\we$ is equivalent on $[0, \rmax(\dom))$ to a power series
with positive coefficients. Then the following properties are equivalent:
\begin{itemize}
  \item [(i)] $\int_0^r \we(t)\, dt \le C \wsec(r)$, $0\le r < \rmax(\dom)$;
  \item [(ii)] The integration operator $\igl: H_\we^\infty(\dom) \to H_\wsec^\infty(\dom)$ is bounded.
\end{itemize}
\end{prop}
\begin{proof}
(i)$\Rightarrow$(ii)
This implication is straightforward and it holds for arbitrary weights $\we, \wsec$.
See, for example, \cite{AT17}.

(ii)$\Rightarrow$(i)
By the hypothesis, there exist $a_k\ge 0$, $k=0,1,\dots$, such that
\[
\we(t) \le \sum_{k=0}^\infty a_k t^k \le C \we(t), \quad 0<t<\rmax(\dom).
\]
Put $g(z) = \sum_{k=0}^\infty a_k z^k$, $z\in\dom$. Then $g\in H_\we^\infty(\dom)$ and
\[
\igl g(r) \ge \int_0^r \we(t)\, dt.
\]
Hence, $\int_0^r \we(t)\, dt \le C \|J\| \wsec(r)$, as required.
\end{proof}

\begin{rem}\label{r_at_disk_bdd}
It is know that a weight $\we$ on $[0,1)$ is equivalent to a power series with positive
coefficients if and only if $\we$ is equivalent to a log-convex weight (see \cite[Proposition~1.3]{Dou15}).
Therefore, Proposition~\ref{p_igl_bdd} for $\dom = \Dbb$ is essentially proved in \cite[Proposition~4]{AT17CM}, where the
corresponding compact operators are characterized under assumption that $\we$ is a log-convex weight.
So, below in the present section we assume that $\dom=\Cbb$.
\end{rem}

\begin{rem}\label{r_at_plane_bdd}
The property of being equivalent to a power series with positive coefficients on $[0, +\infty)$
seems to be more sophisticated than its analog on $[0,1)$.
In particular, equivalence of $\we$ to a log-convex weight on $[0, +\infty)$ does not imply
that $\we$ is equivalent to a power series with positive coefficients on $[0, +\infty)$.
Nevertheless, a constructive characterization of the property in question is known (see \cite{AD17}
for further details and other equivalent properties). What is more interesting,
Corollary~\ref{c_igl_plane} below shows that Proposition~\ref{p_igl_bdd} extends to arbitrary weights on $[0, +\infty)$ via the notion of associated weight.
\end{rem}

\begin{df}
  A weight $\we: [0, +\infty) \to (0, +\infty)$ is called \textsl{essential} if $\we\asymp\wass$.
\end{df}

\begin{cory}\label{c_igl_plane}
Let $\we, \wsec$ be weights on $[0, \infty)$.
Then the following properties are equivalent:
\begin{itemize}
  \item [(i)] $\int_0^r \wass(t)\, dt \le C \wsec(r)$, $0\le r < \infty$;
  \item [(ii)] $\igl: H_\we^\infty(\Cbb) \to H_\wsec^\infty(\Cbb)$ is a bounded operator.
\end{itemize}
\end{cory}
\begin{proof}
Observe that $(\wass)\widetilde{\ }= \wass$. Hence, the weight $\wass$ is essential.
Therefore, $\wass$ is equivalent to a power series with positive coefficients
by \cite[Theorem~2]{AD17}.
So, by Proposition ~\ref{p_igl_bdd}, the corollary holds with $H_{\wass}^\infty(\Cbb)$ in the place
of $H_\we^\infty(\Cbb)$. It remains to recall that $H_{\wass}^\infty(\Cbb) = H_\we^\infty(\Cbb)$ isometrically.
\end{proof}

\subsection{Compact operators}

In this section, we deduce Theorem~\ref{t_igl_cmp} from the following proposition.

\begin{prop}\label{p_igl_cmp}
Let $\we, \wsec$ be weights on $[0, \infty)$.
Assume that $\we$ is equivalent on $[0, \infty)$ to a power series
with positive coefficients. Then the following properties are equivalent:
\begin{itemize}
  \item [(i)] $$\lim_{r\to\infty} \frac{1}{\wsec(r)}\int_0^r \we(t)\, dt  =0;$$
  \item [(ii)] $\igl: H_\we^\infty(\dom) \to H_\wsec^\infty(\dom)$ is a compact operator.
\end{itemize}
\end{prop}
\begin{proof}
(i)$\Rightarrow$(ii)
The implication follows from the analogous one in Proposition~\ref{p_igl_bdd}; see \cite[Proposition~3]{AT17CM}.

(ii)$\Rightarrow$(i)
By the hypothesis, we are given $a_k\ge 0$, $k=0,1,\dots$, such that
\[
2\we(t) \le \sum_{k=0}^\infty a_k t^k \le C \we(t), \quad 0\le t < \infty.
\]
Put
\[
g_n(z) = \sum_{k=n}^\infty a_k z^k, \quad z\in\Cbb.
\]
Clearly, $g_n$ is an entire function.

Firstly, fix an $R>0$. If $z\in \Cbb$ and $|z|\le r < R$, then
\[
g_n(r) \le \frac{r^n}{R^n} g_0(R) \to 0\quad \textrm{as}\ n\to \infty.
\]
Hence, $|g_n(z)|\to 0$ as $n\to \infty$ uniformly on the disk $\{|z|\le r\}$.
So, $g_n$ is a bounded sequence in $H_\we^\infty(\Cbb)$ and $g_n\to 0$ as $n\to \infty$
uniformly on compact subsets of $\Cbb$.
The operator $J$ is assumed to be compact, hence, a standard argument (see, for example, \cite[Proposition~3.1]{CmC95})
guarantees that
\begin{equation}\label{e_tozero}
\|J g_n\|_{H_{\wsec}^\infty(\Cbb)} \to 0 \ \textrm{as}\ n\to \infty.
 \end{equation}

Secondly, we claim that
\begin{equation}\label{e_gn_w}
\int_0^r g_n(t)\, dt \ge \int_0^r \we(t)\, dt\quad \textrm{for all sufficiently large}\ r>0.
\end{equation}
Indeed, put $p_n(t) = \sum_{k=0}^{n-1} a_k t^k$.
Since $\log r = o(\log\we(r))$ as $r\to \infty$,
there exists $t_0>0$ such that
\[
\we(t) \ge 2 p_n(t)\quad \textrm{and}\quad  g_n(t) = g_0(t) -p_n(t) \ge 3\we(t)/2
\]
for all $t\ge t_0$. Thus, \eqref{e_gn_w} holds.

Now, fix an $n\in\Nbb$. Applying \eqref{e_gn_w}, we obtain
\[
\|J g_n\|_{H_{\wsec}^\infty(\Cbb)} \ge \limsup_{r\to \infty}\frac{1}{\wsec(r)}\int_0^r g_n(t)\, dt
\ge \limsup_{r\to \infty}\frac{1}{\wsec(r)}\int_0^r \we(t)\, dt.
\]
Combining the above estimate and \eqref{e_tozero}, we conclude that (ii) implies (i).
\end{proof}

\begin{proof}[Proof of Theorem~\ref{t_igl_cmp}]
Let $\we$ be an arbitrary weight on $[0, +\infty)$.
As shown in the proof of Corollary~\ref{c_igl_plane}, $\wass$ is equivalent to a power series with positive coefficients. So, applying Proposition~\ref{p_igl_cmp} and the isometric equality $H_{\wass}^\infty(\Cbb) = H_\we^\infty(\Cbb)$, we obtain Theorem~\ref{t_igl_cmp}.
\end{proof}

\section{Associated weights}\label{s_assoc}
Given $0<p<\infty$ and a weight $\we: [0, \rmax(\dom)) \to (0, +\infty)$, put
\[
\wasp(t) = \sup\{\mm_p(f, t): \mm_p(f, r)\le \we(r)\ \textrm{for all}\ 0\le r<\rmax(\dom)\}.
\]
For $p=\infty$, the above definition reduces to that of the associated weight $\wass$;
so, we say that the weight $\wasp$ is $p$-associated to $\we$.
The definition of $\wasp$ guarantees that $H_{\wasp}^p(\dom)=\haw^p(\dom)$ isometrically.

The associated weights have proved to be very useful, since they have certain additional regularity.
For $0<p<\infty$, we have a similar phenomenon.

\begin{df}
A function $v: [0, \rmax(\dom)) \to (0, +\infty)$ is called \textsl{log-convex} if $\log v(r)$
is a convex function of $\log r$, $0<r<\rmax(\dom)$, that is,
$\log u (e^x)$ is convex for $-\infty < x < \log \rmax(\dom)$, where $\log(+\infty) = +\infty$.
\end{df}

Given $f\in\hol(\dom)$ and $0<p< \infty$,
the classical Hardy convexity theorem guarantees that $\mm_p (f, r)$, $0\le r < \rmax(\dom)$, is increasing and
log-convex; see \cite{H14, T50}.
So, for each $p\in (0, \infty)$, the function $\wasp(t)$ is a weight; moreover,
$\wasp(t)$ is log-convex as the pointwise supremum of log-convex functions.
In particular, $\wasp^\prime(t)$, the right derivative of $\wasp$, exists for all $0\le t < \rmax(\dom)$.

We will need two results related to $p$-associated weights on $[0,1)$ and $[0,+\infty)$, respectively.

\begin{prop}\label{p_logconv_ass}
  Let $0<p<\infty$ and let $\we$ be a log-convex weight on $[0,1)$.
  Then there exists a constant $C>0$ such that
  \[
  \wasp(r) \le \we(r) \le C \wasp(r), \quad 0\le r< 1.
  \]
\end{prop}
\begin{proof}
Clearly, $\wasp \le \we$. Next, by Theorem~1.1 from \cite{D17bams}, there exists a function $f\in\hol(\Dbb)$
and a constant $C>0$ such that $\we(r)/C \le \mm_p(f,r) \le \we(r)$, $0\le r< 1$; hence, $\we\le C \wasp$.
\end{proof}

\begin{prop}\label{p_logconv_rw}
Let $0<p<\infty$ and let $\we$ be a log-convex weight on $[0,+\infty)$.
  Then there exists a constant $C>0$ such that
  \[
  \we(r) \le C(r+1) \wasp(r), \quad 0\le r< +\infty.
\]
\end{prop}
\begin{proof}
First, assume that $\we(0) = \we(1) =1$.
By assumption, the logarithmic transform $\Phi_\we(t) =\log\we(e^t)$, $-\infty < t <+\infty$,
is a convex function. Put $a_0 = \log\we(0) =0$ and $t_0=0$.
For $n=1,2,\dots$, select $a_n\in\Rbb$ such that $a_n + nt \le \Phi_\we(t)$
and there exists a point $t_n$ such that $a_n + nt_n = \Phi_\we(t_n)$;
we additionally assume that $a_n + nt < \Phi_\we(t)$ for $t>t_n$.

Put
\[
G(t) =\sup_{n\ge 0} (a_n + nt) = \max_{n\ge 0} (a_n + nt), \quad -\infty < t < +\infty.
\]
For $n=0,1,\dots$, we estimate the difference $\Phi_\we(t) - G(t)$ on the interval $[t_n, t_{n+1}]$.
Clearly, we may assume that $t_{n+1}> t_n$.
Let $\Phi_\we^\prime(t)$ denote the right derivative of the convex function $\Phi_\we(t)$.
We have $G^\prime(t) =n$ and $\Phi_\we^\prime(t) \le n+1$ for $t_n \le t < t_{n+1}$; hence,
\[
\Phi_\we(t) - G(t) \le 1\cdot (t-t_n) \le t.
\]
Taking the exponentials, we obtain
\begin{equation}\label{e_we_trop}
\we(e^t) \le e^t \exp(G(t)), \quad 0\le t < \infty.
\end{equation}

Now, let $0<p<\infty$. Fix an $r\ge 1$ and put $t=\log r$. By the definition of $G$,
we have $G(t) = a_n + nt$ for a positive integer $n$.
Then $f_n(z) = e^{a_n} z^n$, $z\in\Cbb$, is a holomorphic monomial such that $\mm_p(f_n, e^t) =\exp(G(t))$.
Thus, by \eqref{e_we_trop},
\[
r\mm_p(f_n, r) \ge \we(r) \ge \mm_p(f_n, r).
\]
In other words, if $\we(0) =\we(1) =1$, then $\we(r) \le r \wasp(r)$, $r\ge 1$.
Therefore, we have the required estimate $\we(r) \le C(r+1) \wasp(r)$, $0\le r< +\infty$,
for an arbitrary log-convex $\we$.

\end{proof}

\section{Differentiation: basic results}\label{s_basic}

\subsection{A necessary condition}
For $p=\infty$, the following result was obtained in \cite[Corollary~2.2]{AT17}.

\begin{prop}\label{p_intoder}
Let $\we$ and $\wsec$ be weights on $[0, \rmax(\dom))$ and let $1\le p<\infty$.
Assume that the differentiation operator maps $\haw^p(\dom)$ into $H_\wsec^p(\dom)$. Then
\[
\wasp^\prime(r) \le C \wsecasp(r), \quad 0\le r < \rmax(\dom).
\]
\end{prop}
\begin{proof}
Let $B(\haw^p)$ denote the unit ball of $\haw^p(\dom)$.
Under assumptions of the proposition, there exists a constant $C>0$
such that
\[
\|f^\prime\|_{H_{\wsecasp}^p}= \|f^\prime\|_{H_\wsec^p} \le C \quad\textrm{for all}\
f\in B(\haw^p).
\]

Let $t>0$ and $r, r+t\in [0, \rmax(\dom))$. Clearly, we have $\mm_p(f,r) \le \wasp(r)$.
Also, by the definition of $\wasp(r+t)$, there exists $f=f(r,t)\in B(\haw^p)$ such that
\[
\mm_p(f, r+t) \ge \wasp(r+t)-t^2.
\]
Since $p\ge 1$, we have
\begin{align*}
  \frac{\wasp(r+t)- \wasp(r) -t^2}{t} & \le \frac{\mm_p(f, r+t) - \mm_p(f,r)}{t} \\
  &\le\left(\frac{1}{2\pi}\int_0^{2\pi} \left| \frac{f((r+t)e^{i\theta})- f(re^{i\theta})}{t}  \right|^p \,d\theta \right)^{\frac{1}{p}}.
\end{align*}
Taking the limit as $t\to 0+$,
we obtain
\[
\wasp^\prime(r) \le \mm_p(f^\prime,r) \le C \wsecasp(r),
\]
since $f\in B(\haw^p)$.
\end{proof}

\subsection{Sufficient conditions}

\begin{lem}[{\cite[Lemma~2.2]{P99}}]\label{l_imeans}
Let $f\in\hol(\dom)$ and $0<p<\infty$. Then there exists a constant $C_p>0$ such that
\[
\mm_p(f^\prime, r) \le  \frac{C_p}{R-r} \max_{0<t<R}\mm_p(f, t)
\]
for all $0<r<R<\rmax(\dom)$.
\end{lem}

\begin{cory}\label{c_suff_disk}
  Let $\we$ and $\wsec$ be weights on $[0,1)$ such that
  \[
  \we\left(\frac{1+r}{2}\right) \le C (1-r) \wsec(r), \quad 0\le r <1.
  \]
Then the differentiation operator $\diff: \haw^p(\Dbb)\to H_\wsec^p(\Dbb)$ is bounded for all $0<p<\infty$.
\end{cory}
\begin{proof}
  Put $R=\frac{1+r}{2}$ in Lemma~\ref{l_imeans} for $\dom=\Dbb$.
\end{proof}

\begin{cory}\label{c_suff_plane}
  Let $\we$ and $\wsec$ be weights on $[0,+\infty)$ such that
  \[
  \we(1+r) \le C \wsec(r), \quad 0\le r <+\infty.
  \]
Then the differentiation operator $\diff: \haw^p(\Cbb)\to H_\wsec^p(\Cbb)$ is bounded for all $0<p<\infty$.
\end{cory}
\begin{proof}
  Put $R=1+r$ in Lemma~\ref{l_imeans} for $\dom=\Cbb$.
\end{proof}

\section{Differentiation: specific target spaces}\label{s_special}

\subsection{Differentiation from $\haw^p(\Dbb)$ into $H_\wsec^p(\Dbb)$ for $\wsec(r)= \we(r)/(1-r)$}

\begin{df}
A weight $\we: [0,1)\to (0,+\infty)$ is called doubling if the exists a constant $d>0$
such that
\[
\sup_{n\in\zz} \frac{\we(1-2^{-n-1})}{\we(1-2^{-n})} \le d <\infty.
\]
We say that $d$ is a doubling constant for $\we$.
\end{df}

\begin{prop}\label{p_suff_disk}
  Let $\we$ be a doubling weight on $[0,1)$. Put
  \[
  \wsec(r) = \frac{\we(r)}{1-r}, \quad 0\le r <1.
  \]
Then the differentiation operator $\diff: \haw^p(\Dbb)\to H_\wsec^p(\Dbb)$ is bounded for all $0<p<\infty$.
\end{prop}
\begin{proof}
Let $d$ be a doubling constant for $\we$. Select $n\in\zz$ such that $2^{-n} \ge 1-r > 2^{-n-1}$.
Then
\[
\we(r) \ge \we (1-2^{-n}) \ge d^{-2} \we(1-2^{-n-2}) \ge d^{-2} \we\left(\frac{1+r}{2} \right).
\]
In other words,
\[
\we\left(\frac{1+r}{2}\right) \le d^2 \we(r) = d^2 (1-r) \wsec(r), \quad 0\le r <1,
\]
hence, Corollary~\ref{c_suff_disk} applies.
\end{proof}

\begin{prop}\label{p_ness_disk}
  Let $1\le p<\infty$ and let $\we$ be a log-convex weight on $[0,1)$. Assume that the differentiation operator
  $\diff: \haw^p(\Dbb)\to H_\wsec^p(\Dbb)$ is bounded, where $\we(r) = (1-r)\wsec(r)$, $0\le r <1$.
  Then $\we$ is a doubling weight.
\end{prop}
\begin{proof}
By Proposition~\ref{p_intoder},
\[
\wasp^\prime(r) \le C \left( \frac{\we(r)}{1-r}\right)^\sim_p \le C \frac{\we(r)}{1-r}, \quad 0\le r <1.
\]
By Proposition~\ref{p_logconv_ass}, $\we\le C\wasp$. Therefore,
\[
\frac{\wasp^\prime(r) (1-r)}{\wasp(r)} \le C, \quad 0\le r <1.
\]
Applying Lemma~2.6 from \cite{AT17} to $\wasp$, we obtain
\[
\frac{\wasp(1-2^{-n-1})}{\wasp(1-2^{-n})} \le C, \quad n=0,1,2,\dots.
\]
Finally, by Proposition~\ref{p_logconv_ass}, the above property holds with $\we$ in the place of $\wasp$;
in other words, $\we$ is doubling.
\end{proof}

For the log-convex weights and $1\le p<\infty$, combining Propositions~\ref{p_suff_disk} and \ref{p_ness_disk},
we obtain the following characterization.

\begin{theorem}\label{t_iff_disk}
  Let $1\le p<\infty$ and let $\we$ be a log-convex weight on $[0,1)$. Put $\wsec(r) = \we(r)/(1-r)$.
  The differentiation operator
  $\diff: \haw^p(\Dbb)\to H_\wsec^p(\Dbb)$ is bounded if and only if $\we$ is doubling.
\end{theorem}

\begin{cory}\label{c_iff_disk}
Let $\we$ be a log-convex weight on $[0,1)$. Put $\wsec(r) = \we(r)/(1-r)$.
If the differentiation operator
  $\diff: \haw^p(\Dbb)\to H_\wsec^p(\Dbb)$ is bounded for some $1\le p \le\infty$,
  then $\diff: \haw^p(\Dbb)\to H_\wsec^p(\Dbb)$ is bounded for all $1\le p \le\infty$.
\end{cory}
\begin{proof}
We apply Theorem~2.8 from \cite{AT17} and Theorem~\ref{t_iff_disk}.
\end{proof}

\subsection{Bounded differentiation operator on $\haw^p(\Cbb)$}
\begin{prop}\label{p_suff_plane_lc}
Let $0<p<\infty$ and let $\we$ on $[0, +\infty)$ be a log-convex weight such that
$\we^\prime(r) \le C\we(r)$, $1\le r <\infty$.
Then the differentiation operator is bounded on $\haw^p(\Cbb)$.
\end{prop}
\begin{proof}
The hypothesis of the proposition guarantees that $\Phi^\prime(x) = (\log\we(e^x))^\prime \le C e^x$, $x\in\Rbb$.
So, considering the interval $[\log r, \log(r+1)]$ and using the convexity of $\Phi$, we obtain
\[
\log\frac{\we(r+1)}{\we(r)} \le C (r+1) \log\frac{r+1}{r} \le C, \quad r\ge 1.
\]
Hence, $\we(r+1)\le C\we(r)$, $r\ge 0$.
It remains to apply Corollary~\ref{c_suff_plane}.
\end{proof}

So, we have an explicit sufficient growth condition for arbitrary weights.

\begin{cory}\label{c_sufflog_plane}
Let $0< p <+\infty$ and let $\we$ be a weight on $[0, +\infty)$.
If $\log\we(r) \le Cr$, $1\le r <+\infty$, then the differentiation operator is bounded on $\haw^p(\Cbb)$.
\end{cory}
\begin{proof}
  The weight $\wasp$ is log-convex and $\log\wasp(r) \le \log\we(r) \le Cr$, $1\le r <+\infty$.
  By \cite[Theorem~2.10]{AT17}, $\log\wasp(r) \le Cr$ implies $\wasp^\prime(r) \le C\we(r)$.
  Since $\haw^p(\Cbb)=H^p_{\wasp}(\Cbb)$ isometrically, it suffices to apply Proposition~\ref{p_suff_plane_lc}
  to $\wasp$.
\end{proof}

\begin{prop}\label{p_ness_plane_lc}
Let $1\le p<\infty$ and let $\we$ be a log-convex weight on $[0, +\infty)$.
Assume that the differentiation operator is bounded on $\haw^p(\Cbb)$.
Then there exists a constant $C>0$ such that
\[
\log\we(r) \le Cr, \quad 1\le r <+\infty.
\]
\end{prop}
\begin{proof}
  By Proposition~\ref{p_intoder},
\[
(\log\wasp)^\prime(r) = \frac{\wasp^\prime(r)}{\wasp(r)} \le C, \quad 0\le r <+\infty.
\]
Integrating the above estimate, we obtain $\log\wasp(r) \le C(r+1)$.
Since $\we$ is log-convex, Proposition~\ref{p_logconv_rw} guarantees that $\we(r) \le C (r+1)\wasp(r)$.
Thus, $\log\we(r) \le C(r+1) + \log C(r+1) \le C(r+1)$, $0\le r < +\infty$, as required.
\end{proof}

Therefore, we have the following characterization for arbitrary weights.

\begin{theorem}\label{t_iff_plane}
Let $1\le p <+\infty$ and let $\we$ be a weight on $[0, +\infty)$.
The following properties are equivalent:
\begin{itemize}
  \item[(i)] $\log\wasp(r) \le Cr$, $1\le r <\infty$;
  \item[(ii)] $\wasp^\prime(r) \le C\we(r)$, $1\le r <\infty$;
  \item[(iii)] The differentiation operator is bounded on $\haw^p(\Cbb)$.
\end{itemize}
\end{theorem}
\begin{proof}
The weight $\wasp$ is log-convex, hence, (i) implies (ii)
by \cite[Theorem~2.10]{AT17}.
Also, (ii)$\Rightarrow$(iii)$\Rightarrow$(i) by Propositions~\ref{p_suff_plane_lc} and \ref{p_ness_plane_lc},
respectively.
\end{proof}

\begin{rem}
If $\we$ is a log-convex weight, then clearly
Theorem~\ref{t_iff_plane} holds with
$\we$ in the place of $\wasp$ in properties (i) and (ii).
\end{rem}

\end{document}